\documentclass[12pt]{amsart}
\usepackage{graphicx, color}
\usepackage{verbatim}
\usepackage{amsmath}
\usepackage{amssymb, mathrsfs}
\usepackage{amsbsy}
\usepackage{amscd}
\usepackage{amsthm}
\usepackage{bm}

\newcommand{\N}{\mathbb{N}}
\newcommand{\R}{\mathbb{R}}
\newcommand{\Z}{\mathbb{Z}}

\newcommand{\set}[1]{\left\{#1\right\}}

\renewcommand{\d}{{\rm{d}}}

\newtheorem{theorem}{Theorem}
\newtheorem*{theorem*}{Theorem}
\newtheorem{lemma}{Lemma}
\newtheorem{proposition}{Proposition}
\newtheorem*{proposition*}{Proposition}

\theoremstyle{definition}
\newtheorem{remark}{Remark}

\numberwithin{equation}{section}

\begin{document}

\title[Directed polymer with unbounded jumps]
{Limiting results for the free energy of directed polymers in 
random environment with unbounded jumps}

\author{Francis Comets}
\address[Francis Comets]
{Laboratoire Probabilit\'es et Mod\'elisation Al\'eatoire, 
Universit\'e Paris Diderot - Paris 7, Paris, France}
\email{comets@math.univ-paris-diderot.fr}

\author{Ryoki Fukushima} 
\address[Ryoki Fukushima]
{Research Institute in Mathematical Sciences, 
Kyoto University, Kyoto, Japan}
\email{ryoki@kurims.kyoto-u.ac.jp}

\author{Shuta Nakajima} 
\address[Shuta Nakajima]
{Research Institute in Mathematical Sciences, 
Kyoto University, Kyoto, Japan}
\email{njima@kurims.kyoto-u.ac.jp}

\author{Nobuo Yoshida}
\address[Nobuo Yoshida]
{Graduate School of Mathematics, Nagoya University, Nagoya, Japan}
\email{noby@math.nagoya-u.ac.jp }

\date{\today}


\begin{abstract}
We study asymptotics of the free energy for the
directed polymer in random environment. 
The polymer is allowed to make unbounded jumps and the environment
is given by Bernoulli variables. 
We first establish the existence and continuity of the free energy 
including the negative infinity value of the coupling constant $\beta$. 
Our proof of existence at $\beta=-\infty$ differs from existing
ones in that it avoids the direct use of subadditivity. 
Secondly, we identify the asymptotics of the free energy at 
$\beta=-\infty$ in the limit of the success probability 
of the Bernoulli variables tending to one. 
It is described by using the so-called time constant of a certain 
directed first passage percolation. 
Our proof relies on a certain
continuity property of the time constant, which is of independent 
interest. 
\end{abstract}

\maketitle

\section{Introduction and main results}

The directed polymer in random environment is a statistical 
physics model of a polymer in disordered solvent. 
In the discrete set-up, the polymer chain is a random walk 
$((X_n)_{n\ge 0}, P)$ on $\Z^d$ starting at the origin 
and the random environment is modelled by independent and identically
distributed random variables 
$((\eta(j,x))_{(j,x)\in \N\times \Z^d}, Q)$. 
We introduce the Hamiltonian $H_n^\eta=\sum_{j=1}^n\eta(j,X_j)$ 
and, for a given inverse temperature $\beta\in\R$, 
define the finite volume Gibbs measure by
\begin{equation}
\d \mu_n^{\eta,\beta}=\frac1{Z_n^{\eta,\beta}}
\exp\{\beta H_n^{\eta}\}\d P,
\label{polymer measure}
\end{equation}
where 
$Z_n^{\eta,\beta}=P[\exp\{\beta H_n^\eta\}]$ 
is the partition function
with $P[\cdot]$ denoting the expectation with respect to $P$.
When $\beta>0$, the polymer is attracted by large values of $\eta$ 
and repelled by negative values. It is known that this interaction 
causes a localization transition depending on 
the law of the random walk~\cite{CY06}. 

A quantity of particular importance in this model is the free energy 
\begin{equation*}
\varphi(\beta)
=\lim_{n\to\infty}\frac{1}{n}\log Z_n^{\eta,\beta}
\end{equation*}
whose existence is usually established by a subadditivity argument.
It is for instance believed that the difference between 
$\varphi(\beta)$ and the so-called annealed free energy characterizes
the localized/delocalized phases. 
See~\cite{CH02,CSY03,CC13,CY13}
for rigorous results in this direction.

\subsection{Zero temperature limits and open paths counting}
One of the main results in the present article is about 
the zero temperature limit of the free energy $\varphi(\beta)$. 
Let us give a few words on the 
{\em motivation}. There has recently been a revival of interest in the 
problem concerning the number of extremal paths in random media 
that dates back to~\cite{Gri83,Dar91}, 
see for example~\cite{CPV08,Yos08,KS10,Lac12,FY12,GGM15} 
for recent works in the directed setup. 
Among others, Garet--Gou\'er\'e--Marchand~\cite{GGM15} have recently 
established the existence of the growth rate of the number of 
open paths in nearest neighbor oriented percolation. 
To be more precise, let 
$N_n$ be a number of open paths of length $n$ starting from 
$(0,0)\in\N\times\Z^d$. Then assuming that the percolation takes 
place with a positive probability, it is proved that
$\lim_{n\to\infty}n^{-1}\log N_n$ exists and is non-random on the event 
of percolation. The main difficulty is that the standard subadditivity 
argument does not work as $\log N_n$ is not well-defined (or
should be defined as $-\infty$) with positive probability, 
making this quantity not integrable. 
One of the motivations of the present work is to propose an approach to
the same problem by considering the zero temperature limit of the 
directed polymer model. Indeed, when the random walk is simple 
nearest neighbor walk and $\eta$ is a Bernoulli variable, 
the above partition function at $\beta=-\infty$ coincides with $(2d)^{-n}N_n$. 
If we are able to prove that the convergence 
\begin{equation*}
\frac1n \log Z_n^{\eta,\beta}\to \frac1n \log N_n-\log 2d
\textrm{ as }\beta\to-\infty
\end{equation*} 
is uniform in $n$ on the event of percolation, then it follows that  
$\lim_{n\to\infty}n^{-1}\log N_n$ exists and is equal to
$\lim_{\beta\to-\infty}\varphi(\beta)+\log 2d$.
In this paper, we carry out this program for random walks with 
stretched-exponential transition probabilities as a test case. 
The unboundedness of jumps simplifies the problem since no percolation 
transition occurs anymore.
However we note that our approach automatically
yields the stronger continuity result of the free energy at 
$\beta=-\infty$. One of the reasons for our rather
special choice of the transition probability is that with this choice, 
the model has a relation to a directed version of the
first passage percolation studied in~\cite{HN97,HN01},
which is interesting in its own right. 
See Theorem~\ref{high-density} below.

We shall comment more on related works in~Subsection~\ref{rel}
after describing our setting and results. 

\subsection{Setting and Results}
Let $(\{X_n\}_{n\in\N}, P_x)$ be the random walk on $\Z^d$ 
starting from $x$ and with the transition probability
\begin{equation*}
 P_x(X_{n+1}=z|X_n=y)=f(|y-z|_1),
\end{equation*}
where $f:\N\cup\{0\}\to (0,1)$ is a function of the form
\begin{equation}
 f(k)=c_1\exp \{-c_2k^\alpha\}, \quad \textrm{ where }  
\alpha>0. \label{f(k)}
\end{equation}
We write $P$ instead of $P_0$ for simplicity.
\begin{remark}
Our choice of the jump law is somewhat arbitrary, and it is tempting to 
replace our specific 
choice with some regular variation assumption on the tail of $\log f(k)$.
It is a purely technical exercise to adapt our method in order to cover 
such cases. To make arguments as transparent as possible we stick to 
this simple law.
\end{remark}
In view of the motivation explained above, we assume 
that $(\{\eta(j,x)\}_{(j,x)\in\N\times\Z^d}, Q)$ 
is independent and identically distributed Bernoulli random variables 
with 
\begin{equation*}
 Q(\eta(0,0)=1)=p\in(0,1).
\end{equation*} 
We define the partition functions at $\beta=-\infty$ by 
\begin{equation*}
 Z_n^{\eta,-\infty}=P(H_n^\eta=0)
\end{equation*}
in addition to the notation introduced before. 
Note that $Z_n^{\eta,-\infty}$ is positive for $Q$-almost every $\eta$,
since the random walk has unbounded jumps. 
It is routine to show  that, 
$Q$-almost surely and for all $\beta\in\R$,
the free energy exists and is equal to the 
second line:
\begin{equation*}
 \begin{split}
\varphi(p,\beta)
&
=\lim_{n\to\infty}\frac{1}{n}\log Z_n^{\eta,\beta}\\
&
=\lim_{n\to\infty}\frac{1}{n}Q[\log Z_n^{\eta,\beta}].
 \end{split}
\end{equation*}
Then, it is plain to see that $\varphi$ is 
non-decreasing 
in $p$ for $\beta>0$, non-increasing in $p$ for $\beta<0$,
non-decreasing and convex in $\beta$, and that
${\varphi}(p,\beta)={\varphi}(1-p,-\beta)+\beta$ for $\beta$ real.
Furthermore, one can show 
by a simple application of the so-called {\em block argument} that 
\begin{equation}
 \lim_{\beta\to-\infty}\varphi(p,\beta)\ge
 \liminf_{n\to\infty}\frac{1}{n}\log Z_n^{\eta,-\infty}>-\infty \;.
\label{lb}
\end{equation}
See Appendix for a proof. 
Our first result shows that the free energy exists and is 
jointly continuous in $(p,\beta)$, including $\beta=-\infty$. 
\begin{theorem}
\label{continuity-FE}
In the above setting with 
$\alpha\in(0,d)$, the limit 
\begin{equation}
\varphi(p,-\infty)=\lim_{n\to\infty}\frac{1}{n}\log Z_n^{\eta,-\infty}
\label{vp(p,-8)}
\end{equation}
exists $Q$-almost surely.
Moreover, the function $\varphi(p,\beta)$ 
is jointly continuous on 
$(0,1]\times [-\infty,\infty)\setminus\{(1,-\infty)\}$.
\end{theorem}
It is possible to show the first part for general 
$\alpha\in(0,\infty)$
by using the subadditive ergodic theorem, as in the proof of 
Theorem~2.1 of~\cite{CMS02},
with the help of the fact that 
\begin{equation*}
 Q[|\log Z_n^{\eta,-\infty}|]<\infty.
\end{equation*}
However, we prove it as a part of the proof of continuity result,
avoiding direct use of the subadditive ergodic theorem at 
$\beta=-\infty$.
As explained above, we think this is of technical importance. 
Note that the above integrability condition may break down even 
for a model where there is no percolation transition. The Brownian 
directed polymer in Poissonian medium with $\beta=-\infty$ is 
such an example, as one can easily check by 
considering the event
that there is a Poissonian trap very close to the origin.

Note that at the exceptional point in Theorem~\ref{continuity-FE},
$\varphi$ should be defined as $\varphi(1,-\infty)=-\infty$. 
It is then natural to ask how $\varphi(p,\beta)$ grows as
$(p,\beta)\to (1,-\infty)$. 
Our next result addresses a directional asymptotics. 
Note that $\varphi (\beta, p)$ exists $Q$-a.s.~for all 
$\alpha >0$, as we have just mentioned.
\begin{theorem}
\label{high-density}
In the above setting with 
{$\alpha\in(0,\infty)$}, 
there exists a constant $\mu_1>0$ such that as
$p \uparrow 1$,
\begin{equation} 
 \varphi (p,-\infty)\sim -c_2\mu_1(1-p)^{-\alpha/d}.
\label{eq:HD}
\end{equation}
The constant $c_2$ comes from \eqref{f(k)}, and $\mu_1$ is defined by 
\eqref{mu_p} with $p=1$.
\end{theorem}
\begin{remark}
If we replace $\eta$ by $1-\eta$ and denote the corresponding free 
energy by $\tilde\varphi(p,\beta)$, 
we can deduce  its asymptotics as $\beta\to+\infty$ and $p\downarrow 0$  from 
Theorem~\ref{continuity-FE} and~\ref{high-density} as follows: 
\begin{equation*}
\lim_{\beta\to+\infty}(\tilde\varphi(p,\beta)-\beta)
\textrm{ exists and asymptotic to }
-c_2\mu_1p^{-\alpha/d}\textrm{ as $p\downarrow 0$}. 
\end{equation*} 
\end{remark}
This kind of symptotics are extensively studied in the continuous time
setting, see Subsection~\ref{rel} below. In the discrete time setting, 
however, this is the first result in the same direction to the best
of our knowledge --- possibly because for the common 
nearest neighbor walk model, the high density asymptotics 
at $\beta=-\infty$ is trivial. 
Moreover, we encounter a new directed first passage
percolation model in identifying the constant $\mu_1$ which is 
interesting in its own right. Let us explain how it comes into play.

The asymptotics~\eqref{eq:HD} has a simple heuristic interpretation. 
When $p$ is close to 1, the sites at which $\eta=0$ have low 
density $1-p$ and hence the random walk has to make a jump of 
order $(1-p)^{-1/d}$ at each step to achieve $H_n^\eta=0$. 
The probability of such a path decays like 
$\exp\{-(1-p)^{-\alpha/d}n\}$ and this explains the $p$-dependent
factor. 
In fact, it turns out that the main contribution to the free 
energy comes from the path which carries the highest probability 
and hence the constant $c_2\mu_1$ corresponds to the growth rate
of the minimal cost for the random walk. 

Note that this minimal cost could in principle depend 
on $p$, {but actually it does not, as we will see in the next
theorem. 
There, we prove the continuity
as $p \uparrow 1$ of the time constant of a certain directed first 
passage percolation, a result of independent interest. 
Denote the (scaled) points where the random walk is allowed to}
go by
\begin{equation*}
 \omega_p=\sum_{(k,x)\in \N\times\Z^d}(1-\eta(k,x))
 \delta_{(k,s_p x)}, 
\end{equation*}
{with the natural scaling factor $s_p=(\log\frac1p)^{1/d}\sim (1-p)^{1/d}$ ($p\uparrow 1$). 
With some abuse of notation we will frequently identify $\omega_p$, and more generally any  point measure, with its support.}
Given a realization of $\omega_p$, we define the passage time from $0$
to $n$ by 
\begin{equation} \label{eq:passtime}
T_n(\omega_p)
= \min\set{\sum_{k=1}^n |x_{k-1}-x_k|^\alpha:
 x_0=0\textrm{ and }\{(k,x_k)\}_{k=1}^n\subset \omega_p}. 
\end{equation} 
Then, a direct application of the 
subadditive ergodic theorem shows that the  limit 
\begin{equation} 
\mu_p=\lim_{n\to\infty}{1 \over n}T_n(\omega_p)
\label{mu_p}
\end{equation} 
exists $Q$-almost surely. The limit $\mu_p$, so-called time constant, is deterministic. 
In these terms, the maximal probability of paths
satisfying $H_n^\eta=0$ is expressed as 
\begin{equation*}
c_1^n\exp\{-c_2s_p^{-\alpha}T_n(\omega_p)\}
=\exp\{-c_2\mu_p(1-p)^{-\alpha/d}n(1+o(1))\}.
\end{equation*}
Now note that $\omega_p$ converges as $p\uparrow 1$ to the Poisson point
process $\omega_1$ on $\N\times \R^d$ whose intensity is the product 
of the counting measure and Lebesgue measure. 
Observe also that definition (\ref{eq:passtime}) makes perfect sense 
when $p=1$, yielding a limit $\mu_1$ in \eqref{mu_p}.
In the next result we claim that the time constant of the Bernoulli model converges to that of the Poisson 
model as $p \uparrow 1$.
\begin{theorem}[{Continuity of the time constant}]
\label{continuity} We have 
$$\lim_{p\uparrow 1}\mu_p=\mu_1.$$
\end{theorem}
\begin{remark}
A similar continuity of the time constant is known for lattice first 
passage percolation in greater generality, see~\cite{Cox80,CK81}
and~(6.9) in~\cite{Kes86}. 
\end{remark}
\subsection{Related works}\label{rel}
The main part of Theorem~\ref{continuity-FE} is the continuity
of $\varphi(p,\beta)$ around $\beta=-\infty$, which is 
the zero temperature asymptotic result for the free energy. 
This type of problems does not seem
to attract much interest in the discrete time setting since in 
some cases the answers are simple.  
For instance, consider the (nearest-neighbor) 
simple random walk model with an i.i.d.~random 
environment with $Q(\eta(0,0)>0)>0$. 
Then it is easy to see that as $\beta\to +\infty$, the free energy
is asymptotic to $\beta$ times the time constant of the directed last 
passage percolation. 
However, if $\eta$ is Bernoulli distributed and we send 
$\beta\to -\infty$, the situation is not so simple. As we mentioned 
at the beginning, the existence of $\varphi(-\infty)$ proved 
in~\cite{GGM15} is already highly nontrivial and the continuity 
as $\beta\to-\infty$ remains an open question at the moment. 

For the continuous time polymer models, the asymptotics of the 
free energy is far from being simple. 
Continuous time random walk models, known under the name of 
parabolic Anderson model, have attracted enormous attention.
Carmona--Molchanov  in the  seminal work~\cite{CM94} initiated 
this line of research. They mainly 
studied the case when the environment is a space--time Gaussian white 
noise and their results include non-matching upper and lower bounds for the 
free energy when the jump rate of the random walk tends to zero. 
Note that this limit is similar to that in Theorem~\ref{high-density} 
in spirit 
since in both cases, the random walk is forced to make more jumps 
than it typically does. 
Shiga~\cite{Shi97} proved similar results for the space--time Poissonian 
environment at $\beta=-\infty$. 
In fact, both~\cite{CM94} and~\cite{Shi97} only proved the
existence of the free energy in the sense of a $L^1$ limit.
These results were later refined and 
extended in~\cite{Mou01,CMS02,CMS05,CM06}
and almost sure existence of the free energy was established
in~\cite{CMS02,CMS05}. 
Finally, the sharp equivalent for the free energy as the jump rate 
vanishes was obtained in \cite{CKM01,CMS02}
in terms of the time constant of a last passage percolation problem. 
Note that for the Gaussian white noise environment, the above asymptotics
is readily translated to the $\beta\to\pm\infty$ limit by using a
scaling identity (see Chapter~IV of \cite{CM94}). 
On the other hand, in the Poissonian environment case, these 
zero temperature limits are of independent interest but 
have not been considered yet. 
In particular, we expect 
that the continuity similar to Theorem~\ref{continuity-FE} holds
when $\beta\to-\infty$. 

Another continuous time polymer model is Brownian directed polymer in 
Poissonian environment introduced by Comets--Yoshida~\cite{CY05}. 
The $\beta\to +\infty$ limit was studied in the same paper, as well as 
$\beta\to-\infty$ for $d\ge 3$ with a specific choice of the other 
parameters. 
It is possible to show by a block argument that the finite volume 
free energy stays bounded as $\beta\to-\infty$ in general 
but, to the best of our knowledge, the existence of the limit at 
$\beta=-\infty$ is not known. 
Later in~\cite{CY13}, the asymptotics as the density of the Poisson 
point process tends to $\infty$ was also studied but only for 
bounded $\beta$, in contrast to Theorem~\ref{high-density} here. 

Finally, we mention that some \emph{solvable models} have been
found recently, see, e.g., 
Moriarty--O'Connell~\cite{MO07}, Amir--Corwin--Quastel
~\cite{ACQ11} and Sepp{\"a}l{\"a}inen~\cite{Sep12}.
{In these models the free energy can be explicitly computed,
thus allowing to study various asymptotics. }
But we refrain from explaining the details of these results 
since such examples have been found only in $(1+1)$-dimension so far 
and also the techniques employed are quite different from ours. 

\subsection{Organization of the paper}
The rest of the paper is organized as follows. 
Section~\ref{zero-temp} is devoted to the proof of 
Theorem~\ref{continuity-FE}. 
For $\beta\in\R$, the continuity is relatively easy and 
the essential part is the proof of continuity around $\beta=-\infty$.
The basic strategy is to introduce a deformation of the path with a
quantitative control of the resulting error.
In Section~\ref{FPP}, we prove Theorem \ref{continuity}, 
as well as a concentration
result which is used in the proof of Theorem~\ref{high-density}. 
Finally, we prove Theorem~\ref{high-density} in
Section~\ref{HD}, by showing that the heuristic computation given 
below Remark 1.4
is indeed correct. There, we closely follow arguments of 
Mountford~\cite{Mou01}. 
\section{Proof of Theorem~\ref{continuity-FE}}\label{zero-temp}
\begin{proof}[Proof of Theorem~\ref{continuity-FE}]
Note first that  continuity in $\beta\in(-\infty,\infty)$ 
follows from convexity of $\varphi(p,\cdot)$. 
Next, we verify the continuity in $p$, locally 
uniformly in $\beta$, cf.~\eqref{eqcont-bdd} below. 
For this purpose, we take arbitrary $0<p<q \le 1$, and introduce another
family of independent and identically distributed Bernoulli variables
$(\{\zeta(j,x)\}_{(j,x)\in\N\times\Z^d}, Q')$
with $Q'(\zeta(0,0)=1)=(q-p)/(1-p)$ and define 
$\check{\eta}=\eta\vee\zeta$. 
Then, $(\{\check{\eta}(j,x)\}_{(j,x)\in\N\times\Z^d},Q\otimes Q')$ 
is a collection of
Bernoulli random variables with success probability $q$ and we 
are going to estimate
\begin{equation*}
Q\otimes Q'\left[
\log Z_n^{\check{\eta},\beta}-\log Z_n^{\eta,\beta}\right]
=Q\otimes Q'\left[
\log \mu_n^{\eta,\beta}
\left[\exp\{\beta H_n^{\check{\eta}-\eta}\}\right]\right],
\end{equation*}
where $\d \mu_n^{\eta,\beta}=(Z_n^{\eta,\beta})^{-1}
\exp\{\beta H_n^{\eta}\}\d P$
is the polymer measure. 
For positive $\beta$, we have by Jensen's inequality that
\begin{equation*}
\begin{split}
0 &\le  Q\otimes Q'\left[
 \log \mu_n^{\eta,\beta}
 \left[\exp\{\beta H_n^{\check{\eta}-\eta}\}\right]\right]\\
&\le
 \log Q\otimes Q'\left[
 \mu_n^{\eta,\beta}
 \left[\exp\{\beta H_n^{\check{\eta}-\eta}\}\right]\right]\\
&\le \log Q\otimes Q'\left[
 \mu_n^{\eta,\beta}
 \left[\exp\{\beta H_n^{\zeta}\}\right]\right]\\
&= \log Q\left[ \mu_n^{\eta,\beta} \left[
Q'\left[\exp\{\beta H_n^{\zeta}\}\right]\right]\right]\\
&= n\log [(e^\beta-1)(q-p)+1].
\end{split}
\end{equation*}
For negative $\beta$, we again use Jensen's inequality {for fixed $\eta$ and $\zeta$} to get
\begin{equation*}
\begin{split}
0 &\ge Q\otimes Q'\left[
 \log \mu_n^{\eta,\beta}
 \left[\exp\{\beta H_n^{\check{\eta}-\eta}\}\right]\right]\\
&\ge
 Q\otimes Q'\left[
 \mu_n^{\eta,\beta}
 \left[\beta H_n^{\check{\eta}-\eta}\right]\right]\\
&\ge 
Q\left[ \mu_n^{\eta,\beta}
 \left[Q'\left[\beta H_n^{\zeta}\right]\right]\right]\\
&= n\beta (q-p). 
\end{split}
\end{equation*}
From these estimates, it follows that for any $M>0$, 
\begin{equation}
\begin{split}
&\lim_{q\downarrow p}
\sup_{|\beta|\le M}|\varphi(q,\beta)-\varphi(p,\beta)|\\
&\qquad\qquad =\lim_{q\downarrow p}
\sup_{|\beta|\le M}\lim_{n\to\infty}\frac{1}{n}
\left|Q\otimes Q'\left[
\log Z_n^{\check{\eta},\beta}-\log Z_n^{\eta,\beta}\right]\right|\\
&\qquad\qquad =0\label{eqcont-bdd}
\end{split}
\end{equation}
and the same holds for $\lim_{p\uparrow q}$.
Combining with the continuity 
in $\beta$, we get the joint continuity on 
$(p,\beta)\in (0,1]\times \R$. 

Now we proceed to the main part of the proof, that is, 
the continuity at $\beta=-\infty$. The following is the key estimate.
\begin{proposition}
\label{Key}
Let $\alpha \in (0,d)$, $p\in(0,1)$ and $\epsilon>0$. Then
there exist $r>0$ and $\beta_0<0$ such that 
for all $q\in [p,p+r]$ and $\beta\in[-\infty,\beta_0]$, 
$Q$-almost surely for all 
sufficiently large $n$, 
\begin{equation}
Z_n^{\eta, \beta}\le e^{\epsilon n}Z_n^{\check{\eta},-\infty}.
\label{key}
\end{equation}
\end{proposition}
Let us first see how to derive Theorem~\ref{continuity-FE} from
this proposition. 
Since the other direction 
$Z_n^{\eta, \beta}\ge Z_n^{\check{\eta},-\infty}$ is obvious, 
we see that 
\begin{eqnarray}
\varphi(p,\beta) -\epsilon 
& \le &\liminf_{n\to\infty}\frac{1}{n}\log Z_n^{\check{\eta},-\infty} 
\nonumber \\
& \le &\limsup_{n\to\infty}\frac{1}{n}\log Z_n^{\check{\eta},-\infty}  
\le \varphi(p,\beta).
\label{conseq:key}
\end{eqnarray}
This in particular implies (by setting $q=p$) that 
the limit \eqref{vp(p,-8)} exists and equals to 
$\lim_{\beta\to-\infty}\varphi(p,\beta)$. Thus, \eqref{conseq:key} reads:
\begin{equation}
\varphi(p,\beta) -\epsilon \le \varphi(q,-\infty) \le \varphi(p,\beta).
\label{conseq2:key}
\end{equation}
Therefore, it also follows from the monotonicity and \eqref{conseq2:key} that
\begin{equation*}
\begin{split}
&\sup\{
|\varphi(p_1,\beta_1)-\varphi(p_2,\beta_2)|
:{p_1, p_2\in [p,p+r], 
\beta_1, \beta_2\in [\beta_0,-\infty]}\}\\
&\quad \le \varphi(p,\beta_0)-\varphi(p+r,-\infty)\\
&\quad \le 2\epsilon.
\end{split}
\end{equation*}
This, together with \eqref{eqcont-bdd}, completes the proof of 
the joint continuity. 
\end{proof}
\begin{proof}
[Proof of Proposition~\ref{Key}]
Hereafter, we denote $Q\otimes Q'$ by $Q$ for simplicity. 
The basic strategy of the proof is to deform the path appearing in 
the sum
\begin{equation} 
Z_n^{\eta,\beta}=\sum_{x_1,\ldots, x_n}\prod_{j=1}^n
f(|x_{j-1}-x_j|_1)
e^{\beta\eta(j,x_j)}\label{sum}
\end{equation}
to a path $x^*$ which does not hit a site with $\check{\eta}(j,x)=1$ 
and compare the above with 
\begin{equation*}
 \sum_{x_1^*,\ldots, x_n^*}\prod_{j=1}^n
 f(|x_{j-1}^*-x_j^*|_1)\le Z_n^{\eta,-\infty},
\end{equation*}
where the sum runs over all paths which appear as a result of
deformation. To establish \eqref{key}, we need
\begin{enumerate}
\item the deformation costs 
$\prod_{j=1}^n\frac{f(|x_{j-1}-x_j|_1)}{f(|x_{j-1}^*-x_j^*|_1)}$
are negligible;
\item not too many paths are deformed to a single path $x^*$.
\end{enumerate}
Let us start the proper proof. We define $x^*$ as follows:
\begin{equation*}
x_k^*=
\begin{cases}
x_k,& \textrm{ if }\check\eta(k,x_k)=0,\\
{\rm argmin}\{{\rm dist}_1(x,\{x:\check\eta(k,x)=0\})\},
& \textrm{ if }\check\eta(k,x_k)=1,
\end{cases}
\end{equation*}
where if there are several candidates in the second case, we 
choose one by a deterministic algorithm. 
To control the costs of deformation, we define 
\begin{equation*}
 d_j(X_j,\check\eta)={\rm dist}_1(X_j,\{x:\check\eta(j,x)=0\}),
\end{equation*}
where $\rm dist_1$ denotes the $l^1$-distance, 
and introduce an auxiliary Hamiltonian 
\begin{equation*}
D_n(X,\check{\eta})
=\sum_{j=1}^n d_j(X_j,\check\eta)^\alpha
\end{equation*}
for $\alpha< 1$ and 
\begin{equation*}
D_n(X,\check{\eta})
=\sum_{j=1}^n d_j(X_j,\check\eta)^\alpha+
|X_{j-1}-X_j|_1^{\alpha-1}(d_{j-1}(X_{j-1},\check\eta)+d_j(X_j,\check\eta))
\end{equation*}
for $1 \le \alpha <d$ with the convention $d_0(X_0,\check\eta)=0$.
When $\alpha<1$, we use the fact 
$(x+y)^\alpha\le x^\alpha+y^\alpha$ for positive $x, y$
to bound the deformation cost at each step as
\begin{equation}
\begin{split}
 \frac{f(|x_{j-1}-x_j|_1)}{f(|x_{j-1}^*-x_j^*|_1)}
&=\exp\{c_2(|x_{j-1}^*-x_j^*|_1^\alpha-|x_{j-1}-x_j|_1^\alpha)\}\\
&\le \exp\{c_2(|x_{j-1}-x_{j-1}^*|_1^\alpha+|x_j-x_j^*|_1^\alpha)\}.
\label{alpha<1}
\end{split}
\end{equation}
In the other case $1\le \alpha<d$, we instead use convexity to get 
\begin{equation}
\begin{split}
&|x_{j-1}^*-x_j^*|_1^\alpha-|x_{j-1}-x_j|_1^\alpha\\
&\quad\le
 \left[|x_{j-1}-x_j|_1+d_{j-1}(x_{j-1},\check\eta)
+d_j(x_j,\check\eta)\right]^\alpha-|x_{j-1}-x_j|_1^\alpha\\
&\quad\le \alpha\left[|x_{j-1}-x_j|_1+d_{j-1}(x_{j-1},\check\eta)
+d_j(x_j,\check\eta)\right]^{\alpha-1}(d_{j-1}(x_{j-1},\check\eta)
+d_j(x_j,\check\eta))\\
&\quad\le
\alpha 2^\alpha|x_{j-1}-x_j|_1^{\alpha-1}(d_{j-1}(x_{j-1},\check\eta)
+d_j(x_j,\check\eta))\\
&\qquad+\alpha 2^{2\alpha}(d_{j-1}(x_{j-1},\check\eta)^\alpha
+d_j(x_j,\check\eta)^\alpha).
\end{split}
\label{alpha>1}
\end{equation}
Hence in both cases, the total cost is bounded as  
\begin{equation*}
 \prod_{j=1}^n\frac{f(|x_{j-1}-x_j|_1)}
{f(|x_{j-1}^*-x_j^*|_1)}\le e^{c_3D_n}
\end{equation*}
for some $c_3>0$.
\begin{lemma}  
\label{small-cost} Let $\alpha\in(0,d)$.
For any $p\in(0,1)$ and $\delta>0$, 
there exists $r\in(0,1)$ and $\beta_0<0$ such that 
for all $q\in[p,p+r]$ and  $\beta\le\beta_0$,
\begin{equation*}
\lim_{n\to\infty}\frac{1}{Z_n^{\eta,\beta}}
P[\exp\{\beta H_n^\eta\}:D_n\le \delta n]= 1, \quad Q\textrm{-a.s.} 
\end{equation*}
\end{lemma}
\begin{proof}
We give a proof only in the case $1\le\alpha<d$ since the other 
case is easier.  
We show that for any $\gamma>0$, one can find 
$\beta_0$ and $r$ such that 
\begin{equation*}
Q[P[\exp\{\beta H_n^\eta+ \gamma D_n\}]]\le 1
\end{equation*}
for all $q\in[p,p+r]$ and $\beta\le\beta_0$. 
Then it readily follows that $Q$-almost surely, 
\begin{equation*}
 P[\exp\{\beta H_n^\eta+ \gamma D_n\}]\le n^2
\end{equation*}
except for finitely many $n \in N$.
If we take 
$\gamma>\lim_{\beta\to -\infty}|\varphi(p,\beta)|/\delta$,
the right-hand side of
\begin{equation*}
P[\exp\{\beta H_n^\eta\}:D_n>\delta n]
\le e^{-\gamma\delta n}P[\exp\{\beta H_n^\eta+ \gamma D_n\}]
\end{equation*}
is $o(Z_n^{\eta,\beta})$ and we are done. 

Let us fix an arbitrary $\gamma>0$ and we write
\begin{equation}
\label{annealed}
\begin{split}
& Q[P[\exp\{\beta H_n^\eta + \gamma D_n\}]]\\
&\quad = P \Biggl[\prod_{j=1}^nQ\bigl[
\exp\bigl\{\beta\eta(j,X_j)
+\gamma d_j(X_j,\check\eta)^{\alpha}\\
&\qquad+ \gamma 
\left(|X_{j-1}-X_j|_1^{\alpha-1}+|X_j-X_{j+1}|_1^{\alpha-1}\right)
d_j(X_j,\check\eta)
\bigr\}\bigr]\Biggr]
\end{split}
\end{equation}
with the convention $|X_n-X_{n+1}|_1=0$. 
We estimate the last $Q$-expectation by distinguishing the cases
according to the value of $\check\eta(j,X_j)$.
First, if $\check\eta(j,X_j)=0$ then all terms in the 
exponential are zero and, by definition,
\begin{equation*}
Q(\check\eta(j,X_j)=0)=1-q. 
\end{equation*}
Second since $\eta(j,X_j)$ and $d_j(X_j,\check\eta)$
are conditionally independent on $\{\check\eta(j,X_j)=1\}$,
we get for general $\xi>0$, 
\begin{equation*}
\begin{split}
&Q\left[
 e^{\beta \eta(j,X_j)+\xi d_j(X_j,\check\eta)}
 1_{\{\check\eta(j,X_j)=1\}}\right]\\
&\quad= Q\left[
 e^{\beta \eta(j,X_j)}1_{\{\check\eta(j,X_j)=1\}}\right]
Q\left[e^{\xi d_j(X_j,\check\eta)}
\big|\check\eta(j,X_j)=1\right]\\
&\quad\le \delta(\beta,r)
Q\left[e^{\xi d_j(X_j,\check\eta)}
\big|\check\eta(j,X_j)=1\right],
\end{split}
\end{equation*}
where 
$\delta(\beta,r)=e^\beta+r\ge e^\beta
+Q(\eta(j,X_j)=0, \zeta(j,X_j)=1)$.
The upper tail of the distribution of $d_j(X_j,\check\eta)$ 
under $Q(\cdot|\check\eta(j,X_j)=1)$ is bounded as
\begin{equation*}
\begin{split}
&Q(d_j(X_j,\check\eta)>r|\check\eta(j,X_j)=1)\\
&\quad = Q(\check\eta(j,x)=1\textrm{ for }
 1\le |x-X_j|_1\le r)\\
&\quad \le q^{cr^d}.
\end{split}
\end{equation*}
As a consequence, we obtain 
\begin{equation}
Q\left[
 e^{\beta \eta(j,X_j)+\xi d_j(X_j,\check\eta)}\right]
\le 1-q+ \delta(\beta,r)e^{\Lambda(\xi)}
\label{cumulant}
\end{equation}
for some regularly varying function $\Lambda$ of index 
$d/(d-1)$ by a standard Tauberian argument. 
(See, for example,~\cite{Kas78}. In fact, it is easy to check 
this fact directly by a Laplace principle type argument.) 
Similarly it also follows from the assumption $\alpha<d$ that 
\begin{equation*}
Q[e^{\beta\eta(j,X_j)+\xi d_j(X_j,\check\eta)^\alpha}]
<1-q+\delta(\beta,r)\Theta(\xi)
\end{equation*}
for some $\Theta(\xi)<\infty$. 

Now we rewrite the exponential in~\eqref{annealed} as
\begin{equation*}
\begin{split}
&\exp\set{\frac{\beta}{3}\eta(j,X_j)
 +\gamma d_j(X_j,\check\eta)^{\alpha}}\\
&\quad \exp\set{\frac{\beta}{3}\eta(j,X_j)+
 \gamma d_j(X_j,\check\eta)
 |X_{j-1}-X_j|_1^{\alpha-1}}\\
&\qquad \exp\set{\frac{\beta}{3}\eta(j,X_j)+
 \gamma d_j(X_j,\check\eta)
 |X_j-X_{j+1}|_1^{\alpha-1}}
\end{split}
\end{equation*}
and apply H\"older's inequality and~\eqref{cumulant} 
to obtain
\begin{equation*}
\begin{split}
&Q\bigl[
 \exp\bigl\{\beta\eta(j,X_j)
 +\gamma d_j(X_j,\check\eta)^{\alpha}\\
&\qquad+ 
 \gamma d_j(X_j,\check\eta)
 \left(|X_{j-1}-X_j|_1^{\alpha-1}+|X_j-X_{j+1}|_1^{\alpha-1}\right)
 \bigr\}\bigr]\\
&\quad\le \left(1-q+\delta(\beta,r)\Theta(3\gamma)\right)^{1/3}\\
&\qquad 
\left(1-q+\delta(\beta,r)
 e^{\Lambda(3\gamma|X_{j-1}-X_j|_1^{\alpha-1})}
\right)^{1/3}\\
&\qquad\left(1-q+\delta(\beta,r)
 e^{\Lambda(3\gamma|X_j-X_{j+1}|_1^{\alpha-1})} 
\right)^{1/3}.
\end{split}
\end{equation*}
We may drop the first factor on the right-hand side 
since it can be made 
smaller than one by letting $\beta$ be close to $-\infty$ and $r$ 
close to zero. 
We then take the product over $1\le j\le n$ 
and $P$-expectation. Due to the independence of 
$\{X_{j-1}-X_j\}_{j=1}^n$ under $P$, the expectation factorizes 
and the term containing $X_{j-1}-X_j$ is
\begin{equation*}
\begin{split}
&P\left[\left(1-q+\delta(\beta,r)
 e^{\Lambda(3\gamma|X_{j-1}-X_j|_1^{\alpha-1})}
 \right)^{2/3}\right]\\
&\quad\stackrel{\textrm{Jensen}}{\le}
\left(1-q+\delta(\beta,r)P\left[
 e^{\Lambda(3\gamma|X_{j-1}-X_j|_1^{\alpha-1})}
 \right]\right)^{2/3}
\end{split}
\end{equation*}
for $2\le j\le n-1$ and for $j\in\{1,n\}$, the exponent 
$2/3$ is replaced by $1/3$. 
In this way, our problem is reduced to checking that 
\begin{equation*}
 P\left[
 e^{\Lambda(3\gamma|X_{j-1}-X_j|_1^{\alpha-1})}
 \right]<\infty.
\end{equation*}
But the function 
$x\mapsto\Lambda(3\gamma x^{\alpha-1})$ is
regularly varying of index $(\alpha-1)\frac{d}{d-1}<\alpha$ {for $\alpha<d$},
hence the above expectation is finite.
\end{proof}
Due to the above lemma, we can restrict the summation~\eqref{sum} 
to paths with $D_n(x,\check\eta)\le\delta n$ and get
\begin{equation*}
\begin{split}
Z_n^{\eta,\beta} &\sim \sum_{x_1,\ldots, x_n: D_n(x,\check\eta)\le\delta n}
 \prod_{j=1}^n f(|x_{j-1}-x_j|_1)
 e^{\beta\eta(j,x_j)}\\
&=\sum_{x_1,\ldots, x_n: D_n(x,\check\eta)\le\delta n}
\prod_{j=1}^nf(|x_{j-1}^*-x_j^*|_1)
\left[\frac{f(|x_{j-1}-x_j|_1)}{f(|x_{j-1}^*-x_j^*|_1)}
e^{\beta\eta(j,x_j)}\right]\\
&\le e^{c_3\delta n}\sum_{y_1,\ldots, y_n: H_n(y,\check\eta)=0}
\#\{x: x^*=y, D_n(x,\check\eta)\le\delta n\}
\prod_{j=1}^n f(|y_{j-1}-y_j|_1).
\end{split}
\end{equation*}
We are left with estimating the number of paths which are 
deformed to a fixed path. 
\begin{lemma}
\label{counting}
There exists a function $\chi(\delta)\to 0$ as $\delta\downarrow 0$ 
such that for any fixed path 
$(y_1, \ldots, y_n)\in (\Z^d)^n$, 
\begin{equation*}
\#\{x: x^*=y, D_n(x,\check\eta)\le\delta n\}
\le \exp\{\chi(\delta)n\}.
\end{equation*}
\end{lemma}
\begin{proof}
We write $z_j=x_j-y_j$. Then it suffices to bound 
\begin{equation*}
\begin{split}
&\#\{(z_j)_{j=1}^n: 
|z_1|_1^\alpha+\cdots+|z_n|_1^\alpha \le \delta n\}\\
&\quad \le e^{\lambda\delta n}
\sum_{z:\,|z_1|_1^\alpha+\cdots+|z_n|_1^\alpha \le \delta n}
e^{-\lambda (|z_1|_1^\alpha+\cdots+|z_n|_1^\alpha)} 
\quad (\lambda>0)\\
&\quad\le \left(\sum_{z \in \Z^d}
e^{\lambda\delta-\lambda|z|_1^\alpha}\right)^n.
\end{split}
\end{equation*}
By taking $\lambda=\delta^{-1/2}$, we find that the right-hand side
is $(1+o(1))^n$ as $\delta\downarrow 0$.
\end{proof}

Combining the above arguments, we can find $r\in(0,1)$ and 
$\beta_0<0$ such that for any $q\in[p,p+r]$ and $\beta<\beta_0$,
\begin{equation*}
\begin{split}
Z_n^{\eta,\beta} &\le
 e^{\epsilon n}\sum_{y_1,\ldots, y_n: D_n(y,\check\eta)=0}
 \prod_{k=1}^n f(|y_{j-1}-y_j|_1)\\
&=e^{\epsilon n}Z_n^{\check\eta,-\infty}
\end{split}
\end{equation*}
for all sufficiently large $n\in\N$. 
\end{proof}

\section{A directed first passage percolation}
\label{FPP}
In this section, we prove Theorem~\ref{continuity}. 
We also prove a concentration bound 
for the passage times,
which is an important ingredient in the proof of 
Theorem~\ref{high-density}. 

For further use, we start by introducing a special realization of $\eta$: recalling that $\eta=\eta_p$ depends in fact on $p$, we define a {\em coupling} of $\eta_p$
for all values of $p \in (0,1)$ as follows.
Let $(Q, \omega_1)$ be the Poisson point process on $\N\times \R^d$ 
whose intensity is the product of the counting measure and Lebesgue measure, and define, for $p \in (0,1)$,
\begin{equation}
 \eta(k,x){= \eta_p(k,x)}=1_{\{\omega_1({\{k\}\times} s_p(x+[0,1)^d))=0\}}
\label{coupling}
\end{equation}
with $s_p=(\log\frac1p)^{1/d} $ the scaling factor.  {Note that $s_p \in (0,\infty)$ and $s_p \to 0$ as $p \uparrow 1$}. 
Let us also introduce 
\begin{equation*}
 \omega_p=\sum_{(k,x)\in \N\times\Z^d}(1-{\eta_p}(k,x))
 \delta_{(k,s_p x)}
\end{equation*}
which vaguely converges to $\omega_1$, $Q$-almost surely {as $p \uparrow 1$}. 
Hereafter, we sometimes identify $\omega_p$ with its support by 
abuse of notation. 
For $0< p \leq 1$,  recall the definition of the passage time from $0$ to $n$, 
\begin{equation*}
T_n(\omega_p)
= \min\set{\sum_{k=1}^n |x_{k-1}-x_k|^\alpha:
 x_0=0\textrm{ and }\{(k,x_k)\}_{k=1}^n\subset \omega_p},
\end{equation*} 
and recall that, by the subadditive ergodic theorem, 
the following limits exist and are equal:
\begin{equation*}
\mu_p=\textrm{a.s.-}\lim_{n\to\infty}{1 \over n}T_n(\omega_p)= 
\inf_{n\in\N} {1 \over n} Q[T_n(\omega_p)]=\lim_{n\to\infty} {1 \over n}Q[T_n(\omega_p)].
\end{equation*} 

\begin{proof}[Proof of Theorem~\ref{continuity}] 
We have the following comparison for the passage times from which 
the result readily follows:
\begin{equation}
\begin{split}
 & T_n(\omega_1) \le (1+\delta_1)T_n(\omega_p)+\delta_2n,\\
 & T_n(\omega_p) \le (1+\delta_1)T_n(\omega_1)+\delta_2n, 
\end{split}
\label{comparison}
\end{equation}
where $\delta_1, \delta_2\to 0$ as {$p\uparrow 1$}. 
We only prove the first one since the argument for the other is
the same. Let $(\pi_n(m))_{m=0}^n$ be a minimizing path for 
$T_n(\omega_p)$. Then, by definition, each $\pi_n(m)+[0,s_p)^d$
contains a point of $\omega_1$. Thus we can find another path
$\{\pi_n'(m)\}_{m=0}^n$ such that  
\begin{equation*}
\pi_n'(0)=0, \pi_n'(m)\in \omega_1 \textrm{ and }
|\pi_n(m)- \pi_n'(m)|_1 \le ds_p
\end{equation*}
for $1\le m\le n$. 
Then, we have 
\begin{equation*}
|\pi_n'(m-1)- \pi_n'(m)|_1
\le |\pi_n(m-1)- \pi_n(m)|_1+ 2ds_p
\end{equation*}
and together with an elementary inequality
\begin{equation*}
(t+s)^\alpha\le
\begin{cases}
t^\alpha+s^\alpha,& \alpha\le 1, \\
(1+s)^{\alpha-1}(t^\alpha+s),& \alpha>1,  
\end{cases}
\end{equation*}
where the second one is obtained by applying convexity 
to $(\frac{1\cdot t+s\cdot 1}{1+s})^\alpha$,
we get
\begin{equation*}
\begin{split}
T_n( \omega_1)
&\le \sum_{m=1}^n|\pi_n'(m-1)- \pi_n'(m)|_1^\alpha\\
&\le 
\begin{cases}
\sum_{m=1}^n|\pi_n(m-1)-\pi_n(m)|_1^\alpha+(2ds_p)^\alpha n,&\alpha\le 1,\\
(1+2ds_p)^{\alpha-1}\left(\sum_{m=1}^n|\pi_n(m-1)-
\pi_n(m)|_1^\alpha+2ds_pn\right),&\alpha>1.
\end{cases}
\end{split}
\end{equation*}
Since $s_p$ tends to zero as {$p\uparrow 1$}, we are done.
\end{proof}

Our second main result in this subsection is the lower tail estimate
of the passage time distribution. 
\begin{proposition}
\label{concentration}
There exist positive constants $C_1$, $C_2$ and $\lambda\in (0,1)$ 
such that for any $n\in\N$,
\begin{equation}
Q\left(T_n(\omega_1)-n\mu_1<-n^{1-\lambda}\right)
\le C_1\exp\left\{-C_2
n^\lambda \right\}.
\label{concentration3}
\end{equation}
\end{proposition}
\begin{proof}
We fix a small $\theta>0$ and define 
 \begin{equation*}
 \bar{\omega}=\omega+\sum_{(k,x)\in \N\times n^\theta\Z^d}
 1_{\{\omega(\{k\}\times (x+[0,n^\theta)^d))=0\}}\delta_{(k,x)},
 \end{equation*} 
that is, 
when we find a large vacant box, we add an $\omega$-point 
artificially at a corner.
This modification provides a uniform bound for the passage time
\begin{equation*}
 \sup_{\omega}T_n( \bar\omega)\le 
d^\alpha n^{1+\alpha\theta}
\label{trivial-bd}
\end{equation*}
since there is a path whose all jumps are bounded by $dn^\theta$. 
We also have the following upper 
tail estimate. 
\begin{lemma}
\label{greedy}
There exists $C_0>0$ such that for all $n\in\N$ and $m>C_0n$,
\begin{equation}
 Q(T_n(\omega_1)>m)\le \exp\{-m^{1\wedge {d \over \alpha}}/C_0\}. 
\label{upper-tail}
\end{equation} 
\end{lemma}
\begin{proof}
Note that $T_n(\omega_1)$ is bounded by the passage 
time of the {\em greedy} path $\{(k,x_k)\}_{k\in\N}$ which is inductively
constructed by minimizing the distance to points in the next section,
that is, $x_0=0$ and
\begin{equation*}
 x_k={\rm argmin}\{|x_{k-1}-x|_1: (k, x)\in\omega_1\}.
\end{equation*} 
The passage time of such a path is nothing but the sum of 
independent random variables with the same distribution as 
${\rm dist}((0,0), \omega_1|_{\{0\}\times\R^d})^\alpha$.
One can bound its tail as
\begin{eqnarray*}
Q({\rm dist}((0,0),
\omega_1|_{\{0\}\times\R^d})^\alpha\ge r)
&=&Q\big(\omega_1|_{\{0\}\times\R^d}(B_{l^1}(0,r^{1/\alpha}))=0)\big)\\ 
&=& \exp\set{-cr^{d/\alpha}}
\end{eqnarray*}
for some $c>0$. 
Our assertion follows from this and a well known result 
for the large deviation of sums of independent random variables,
for which we refer to~\cite{Nag79}. 
\end{proof}
Next, we show that $T_n(\omega_1)$ and $T_n(\bar\omega_1)$ 
are essentially the same. 
\begin{lemma}
\label{approx}
There exists $C_3>0$ such that for sufficiently large $n\in\N$, 
\begin{equation*}
\begin{split}
&\max\{Q(T_n( \omega_1)\neq T_n( \bar\omega_1)),
 Q[|T_n( \omega_1)-T_n( \bar\omega_1)|]\}\\
&\qquad \qquad \qquad \qquad \le \exp\{-C_3n^{d\theta}\}.
\end{split}
\end{equation*}
\end{lemma}
\begin{proof}
Thanks to Lemma~\ref{greedy}, we know that 
$T_n(\omega_1)\le C_0n$ with probability greater than
$1-\exp\{-n^{1\wedge \frac{d}{\alpha}}/C_0\}$. 
Under this condition, all the minimizing 
paths for 
$T_n(\omega_1)$ stay inside 
$\mathcal{C}_n:=[0,n]\times [-C_0^{1/\alpha}n^{1+1/\alpha},
C_0^{1/\alpha}n^{1+1/\alpha}]^d$. 
Indeed, if any minimizing path exits $\mathcal{C}_n$, then
it must make a jump larger than $C_0^{1/\alpha}n^{1/\alpha}$ and hence
its passage time is larger than $C_0n$. 
Since $T_n(\bar\omega_1)\le T_n(\omega_1)$, the same applies
to {minimizing} paths for $T_n(\bar\omega_1)$.
This space-time region contains only polynomially many 
boxes of the form $\{k\}\times (x+[0,n^\theta)^d)$ and each of them
is vacant with probability $\exp\{-cn^{d\theta}\}$.
Thus it follows that 
\begin{equation*}
 Q(\omega_1=\bar\omega_1\textrm{ in }\mathcal{C}_n)
\ge 1-\exp\{-cn^{d\theta}/2\}
\end{equation*}
for large $n$. 
Since $T_n( \omega_1)= T_n( \bar\omega_1)$
on the event 
\begin{equation*}
 \{T_n(\omega_1)\le C_0n \textrm{ and }
\omega_1=\bar\omega_1\textrm{ in }\mathcal{C}_n\},
\end{equation*}
we get the desired bound on $Q(T_n(\omega_1)\neq T_n(\bar\omega_1))$. 

As for the $L^1(Q)$ distance, we use the Schwarz inequality to obtain
\begin{equation*}
\begin{split}
&Q[|T_n( \omega_1)-T_n( \bar\omega_1)|]\\
&\quad \le Q\left[(T_n( \omega_1)-T_n( \bar\omega_1))^2\right]^{1/2}
Q(T_n( \omega_1)\neq T_n( \bar\omega_1))^{1/2}.
\end{split}
\end{equation*}
The first factor on the right-hand side 
is of $O(n)$ as $n\to \infty$ due to Lemma~\ref{greedy}.
\end{proof}
We proceed to a lower tail estimate for $T_n(\bar\omega_1)$. 
Let $\bar\omega_1^{(m)}$ be the point process obtained by 
replacing its $\{m\}\times\R^d$-section by 
$\bar\omega'$ which is the modification of 
another configuration $\omega'$. 
We are going to use the so-called entropy method
(Theorem~6.7 in~\cite{BLM13}) and it 
requires a bound on
\begin{equation}
\sum_{m=1}^n\left(\sup_{\omega'}T_n(\bar\omega_1^{(m)})
-T_n(\bar\omega_1)\right)^2.
\label{ES}
\end{equation}
Let us first assume $\alpha\ge 1$ 
and let $\{\pi_n(m)\}_{m=0}^n$ be 
a minimizing path for $T_n(\bar\omega_1)$.
As we can find a point in $\omega'|_{\{m\}\times\R^d}$ within the 
distance $dn^\theta$ to $\pi_n(m)$, 
\begin{equation*}
\begin{split}
&\sup_{\omega'}T_n(\bar\omega_1^{(m)})
 -T_n(\bar\omega_1)\\
 &\quad\le
{\alpha}(|\pi_n(m-1)-\pi_n(m)|_1+dn^\theta)^{\alpha-1} dn^\theta
{1_{\{ m \ge 1\}} }\\
 &\qquad+{\alpha}(|\pi_n(m)-\pi_n(m+1)|_1+dn^\theta)^{\alpha-1} dn^\theta
{1_{\{ m \le n-1\}}}.
\end{split}
\end{equation*}
Furthermore, the {\em a priori} bound
\begin{equation*}
T_n(\bar\omega_1)=\sum_{m=1}^n|\pi_n(m-1)-\pi_n(m)|_1^\alpha
\le d^\alpha n^{1+\alpha\theta} 
\end{equation*}
yields the following bound on the numbers of large jumps
\begin{equation*}
 \#\{m\le n\colon |\pi_n(m-1)-\pi_n(m)|_1\ge n^{k\theta}\}
\le {d^\alpha}n^{1-(k-1)\alpha\theta}
1_{\{k \le \frac{1}{\alpha\theta}+2\}}.
\end{equation*}
Thus by dividing the sum in~\eqref{ES} according to the indices with
jump size falling in $[n^{k\theta}, n^{(k+1)\theta})$, 
we can bound it, up to a multiplicative constant, by
\begin{equation*}
\begin{split}
\sum_{k\le \frac{1}{\alpha\theta}+2}
n^{1-(k-1)\alpha\theta}n^{2(k+1)\theta(\alpha-1)+2\theta}
{=n^{1+3\alpha\theta}}\sum_{k\le \frac{1}{\alpha\theta}+2}
n^{(\alpha-2)\theta k}.
\end{split}
\end{equation*}
It is simple to check that the right-hand side is bounded
by $n^{\rho}$ with $\rho<2$ when $\theta$ is sufficiently small.
Then, Theorem~6.7 in~\cite{BLM13} yields
\begin{equation*}
Q\left( T_n(\bar\omega_1)-Q[T_n(\bar\omega_1)]<-n^{1-\lambda}
\right)
\le \exp\{-C_2 n^{2-\rho-2\lambda}\}.
\end{equation*}
Lemma~\ref{approx} shows that this remains valid with $\bar\omega_1$
replaced by $\omega_1$ and $\exp\{-C_3n^{d\theta}\}$ added to the 
right-hand side. Finally, since 
$\mu_1=\inf_nn^{-1}Q[T_n(\omega_1)]$, we can further replace
$Q[T_n(\omega_1)]$ by $n\mu_1$ and arrive at
\begin{equation*}
Q( T_n(\omega_1)-n\mu_1<-n^{1-\lambda})
\le \exp\{-C_2n^{2-\rho-2\lambda}\}+\exp\{-C_3n^{d\theta}\}.
\end{equation*}
Choosing $\lambda>0$ small, we get the desired bound. 

The case $\alpha<1$ is simpler since we readily get
$\sup_{\omega'}T_n(\bar\omega_1^{(m)})-T_n(\bar\omega_1)
\le 2d^\alpha n^{\alpha\theta}$ uniformly in $m$ 
just as in~\eqref{alpha<1}.
\end{proof}

\section{Proof of Theorem~\ref{high-density}}
\label{HD}
In this section, we continue to assume that $\eta$ is realized
as in~\eqref{coupling} in the previous section. 
Recall also that we defined $s_p=(\log\frac{1}{p})^{1/d}$, which is
asymptotic to $(1-p)^{1/d}$ as {$p\uparrow 1$.}
{
The positivity of $\mu_1$ can proved by essentially the same 
argument as in the upper bound: 
see Remark~\ref{positivity} below. 
Let us first complete the proof of~\eqref{eq:HD} assuming it. 
}
\begin{proof}
[Lower bound]
Let $\pi_n$ be a 
minimizing path for $T_n(\omega_p)$. Then obviously,
\begin{equation*}
\begin{split}
Z_n^{\eta,-\infty}&=P(H_n^{\eta}=0)\\
&\ge P(X_k=\pi_n(k)\textrm{ for all }1\le k\le n)\\
&= c_1^n\exp\set{-c_2s_p^{-\alpha}T_n(\omega_p)}
\end{split}
\end{equation*}
and hence
\begin{equation*}
\varphi(p,-\infty)\ge -c_2s_p^{-\alpha}\mu_p+\log c_1.
\end{equation*}
By letting $p\uparrow 1$ and using Theorem~\ref{continuity},
we get the desired lower bound.
\end{proof}
The upper bound is more laborious since we have to show that
the number of paths makes negligible contribution. 
We closely follow the argument of Mountford in~\cite{Mou01}. 
\begin{proof}
[Upper bound]
Let 
{$M=(\alpha+2)/\alpha$ and} 
define a {face-to-face} passage time
\begin{equation*}
\Phi_R(\omega_p)
=\inf\Biggl\{\sum_{i=1}^R |x_{i\!-\!1}\!-\!x_i|_1^\alpha: 
{|x_0|_{\infty}\le} R^M\textrm{ and }\\   
(i, x_i)\!\in \!\omega_p
 \textrm{ for }1\!\le \! i\!\le\! R\Biggr\}
\end{equation*}
for $R\in\N$. We fix $\epsilon>0$ and say that 
$(k, x)\in \N\times 2\Z^d$ is $\epsilon$-{\em good} if the 
following two conditions hold:
\begin{enumerate}
\item $\Phi_R(\omega_p-(k,R^Mx)) \ge (\mu_1-\epsilon)R$;
\item $\max_{k+1 \le l \le k+R} 
 \omega_p(\{l\}\times (R^Mx+[-2R^M, 2R^M]^d))\le 4^{d+1}R^{dM}$,
\end{enumerate}
{where $\omega_p-(k,R^Mx)$ is the translation of $\omega_p$ 
regarded as a set.
}
Our basic strategy is to show that: (1) if the polymer,
scaled by a factor of $s_p R^{-M}$, comes close 
to an $\epsilon$-good point, then it costs at least 
$\exp\{-(\mu_1-\epsilon)R\}$ to survive the next $R$-duration; 
(2) most of the times in $\{jR\}_{j=1}^{[n/R]}$, the polymer
is close to an $\epsilon$-good point with high probability.
\begin{lemma}
\label{good-is-typical} 
There exists $p_0(\epsilon)\in(0,1)$ such that 
\begin{equation*}
\lim_{R\to\infty}Q((k,x)\textrm{ is $\epsilon$-good})= 1
\end{equation*}
uniformly in $p\in[p_0(\epsilon),1]$ and $(k,x)\in\N\times 2\Z^d$.
\end{lemma}
\begin{proof}
By translation invariance, we may assume $(k,x)=(0,0)$ without loss
of generality. 
Note also that the probability of
\begin{equation*}
\begin{split}
E_R=\left\{\omega_1\colon\max_{y\in [-R^M,R^M]^d\cap\Z^d}
 T_R( \omega_1-(0,y))\le C_0R \right\}
\end{split}
\end{equation*}
tends to one as $R\to\infty$ by Lemma~\ref{greedy}. 
On this event, we know from \eqref{comparison} that 
\begin{equation}
T_R( \omega_p-(0,y))<T_R( \omega_1-(0,y))+\epsilon R\le (C_0+\epsilon)R
\label{comparison2}
\end{equation} 
for all $p$ close to one. 
As a consequence, all the minimizing paths for 
$T_R( \omega_p-(0,y))$, that is, the passage time 
from $(0,y)$ to $\{R\}\times \R^d$, make jumps of size at most 
a constant multiple of $R^{1/\alpha}$. 
Then by using the mean value theorem, one can check that
\begin{equation}
\Phi_R(\omega_p)-
\min_{y\in [-R^M,R^M]^d\cap\Z^d}T_R( \omega_p-(0,y))\ge d^\alpha
\vee (cR^{(\alpha-1)/\alpha})
\label{diff}
\end{equation}
for some $c>0$, since the difference comes only from the starting points. 

Thus we can bound  
\begin{equation*}
\begin{split}
&Q(\{\Phi_R(\omega_p) \le (\mu_1-2\epsilon)R\}\cap E_R)\\
&\quad \le 
Q\left(\min_{y\in [-R^M,R^M]^d\cap\Z^d}
T_R( \omega_1-(0,y)){+}d^\alpha \vee (cR^{(\alpha-1)/\alpha})
\le (\mu_1-{2}\epsilon)R\right)\\
&\quad\le \sum_{y\in [-R^M,R^M]^d\cap\Z^d}
Q\big(T_R(\omega_1-(0,y))
\le (\mu_1-\epsilon)R\big)\\
&\quad \le (2R^M+1)^d
 C_1\exp\left\{-C_2R^\lambda \right\}
\end{split}
\end{equation*}
for sufficiently large $R$, where we have used~{\eqref{diff}} in
the first inequality. 

On the other hand, a simple large deviation estimate shows that
there is $c>0$ such that for any $l\in\N$, 
\begin{equation*}
Q\left(\omega_p(\{l\}\times [-2R^M, 2R^M]^d)> 4^{d+1}R^{dM}\right)
\le \exp\set{-cR^{dM}}
\end{equation*}
and summing over $l\in \{1,2,\ldots,R\}$, we get
\begin{equation*}
 Q\left(\max_{1 \le l \le R} 
 \omega_p(\{l\}\times [-2R^M, 2R^M]^d)> 4^{d+1}R^{dM}\right)\to 0
\end{equation*} 
as $R\to\infty$.
\end{proof}
Let us write $C_p(x)=s_p^{-1}(R^Mx+[-R^M,R^M]^d)$ for shorthand. 
\begin{lemma}
\label{good-is-good}
For sufficiently large $R\in\N$, there exists 
$p_1(R,\epsilon)>0$
such that if $p\in [{p_1(R,\epsilon)},1)$ and 
$(k,x)$ is $\epsilon$-good, then
\begin{equation*}
 \begin{split}
 &\sup_{y\in C_p(x)}
 P(\eta(l,X_l)=0 \textrm{ for all }l\in\{k+1,\ldots, k+R\}
|X_k=y) \qquad\qquad\\
 &\qquad \qquad\qquad\qquad \le \exp\set{-c_2s_p^{-\alpha}(\mu_1-2\epsilon)R}.
 \end{split}
\end{equation*}
\end{lemma}
\begin{proof}
We again assume that $(k,x)=(0,0)$ without loss of generality. 
We first prove 
\begin{equation}
\sup_{y\in C_p(0)}
{P_y\left(\max_{1\le l\le R}|X_l|_\infty\ge
 2s_p^{-1}R^M\right)}
\le \exp\set{- C_5s_p^{-\alpha}
R^2}
\label{RW-tail}
\end{equation}
so that we may assume the contrary. 
When $\alpha\le 1$, one can readily check that 
\begin{equation*}
\begin{split}
\sup_{y\in C_p(0)}
{P_y\left(\max_{1\le l\le R}|X_l|_\infty\ge 2s_p^{-1}R^M\right)}\!\!\!\!\!\!
&\quad \le
P\left(\max_{1\le l\le R}|X_l|_\infty\ge s_p^{-1}R^M\right)\\
&\quad\le P\left(\sum_{j=1}^R
|X_{j-1}-X_{j}|_1^\alpha\ge 
s_p^{-\alpha}R^{\alpha+2}\right).
\end{split}
\end{equation*}
Since our assumption on the transition probability implies 
\begin{equation}
\label{alpha-Laplace}
C_6:=P\left[\exp\set{\frac{c_2}{2}|X_1|_1^\alpha}\right]\in (1,\infty),
\end{equation}
Chebyshev's inequality yields
\begin{equation*}
\textrm{LHS of \eqref{RW-tail}}\le 
\exp\set{-\frac{c_2}{2}s_p^{-\alpha}R^{\alpha+2}
 +R\log C_6}.
\end{equation*}
For $\alpha>1$, we use Jensen's inequality to get
\begin{equation*}
\sup_{y\in C_p(0)}
{P_y\left(\max_{1\le l\le R}|X_l|_\infty\ge 2s_p^{-1}R^M\right)}
\le P\left(R^{\alpha-1}\sum_{j=1}^R
|X_{j-1}\!-\!X_{j}|_1^\alpha\ge 
s_p^{-\alpha}R^{\alpha+2}\right).
\end{equation*} 
With the help of~\eqref{alpha-Laplace}, the rest of the proof 
is similar to the above. 
 
Thanks to the condition~(i), every path satisfying 
$H^\eta_R(X)=0$ has probability at most 
\begin{equation*}
c_1^R\exp\set{-c_2s_p^{-\alpha}(\mu_1-\epsilon)R}
\end{equation*}
under $P(\cdot|X_0=y)$. On the other hand, condition~(ii) ensures that 
there are at most $(4^{d+1}R^{dM})^R$ such paths which, in addition,
stay inside $[0,R]\times s_p^{-1}[-2R^M, 2R^M]$. Therefore we have
\begin{equation*}
 {P_y\bigl(H_R^\eta=0, \max_{1\le l\le R}|X_l|< 2s_p^{-1}R^M\bigr)}
  \le (c_14^{d+1}R^{dM})^R
\exp\set{-c_2s_p^{-\alpha}(\mu_1\!-\!\epsilon)R}
\end{equation*}
{and since $s_p$ tends to zero as $p\uparrow 1$,} the assertion follows. 
\end{proof}
Let 
$\psi_{\epsilon}(k,x)=c_2(\mu_1-2\epsilon)
1_{\{({kR}, x)
\textrm{ is $\epsilon$-good}\}}$ and 
\begin{equation*}
\Gamma=\set{\gamma=(j,\gamma_j)_{j\in\Z_+}: 
\gamma_0=0, \gamma_j\in 2\Z^d}.
\end{equation*}
For $\gamma\in\Gamma$ and an integer $v\ge 1$, we define
\begin{equation*}
 J_v(\gamma)=\sum_{j=0}^{v-1}
\max\set{\psi_\epsilon(j,\gamma_j),  
C_5 R
(|\gamma_j-\gamma_{j+1}|_\infty-1)_{+}^\alpha}.
\end{equation*}
\begin{lemma}
Let $R$ and $p$ be as in
Lemma~{\ref{good-is-typical} and}~\ref{good-is-good}. 
Then for any $v\ge 1$ and $\gamma\in\Gamma$, 
 \begin{equation*}
 \begin{split}
 &P\left(H_{vR}^\eta=0\textrm{ and }
 X_{jR}\in C_p(\gamma_j)
 \textrm{ for }j=1,\ldots,v\right)\\
 &\qquad \qquad \quad\le
 \exp\set{-s_p^{-\alpha}J_v(\gamma)R}.
 \end{split}
 \end{equation*}
\end{lemma}
\begin{proof}
We use Markov property at times $R,2R, \ldots, {(v-1)R}$ to bound 
the left-hand side by
\begin{equation*}
\begin{split}
{\prod_{j=0}^{v-1}}\sup_{y\in C_p(\gamma_j)}
{P_y\left(H_R^{\theta_{jR}\eta}=0 \textrm{ and }
X_R\in C_p(\gamma_{j+1})\right)},
\end{split}
\end{equation*}
where $\theta_k$ ($k\in\N$) is the time-shift operator acting
on the space of environments. 
By Lemma~\ref{good-is-good}, it immediately follows that
\begin{equation*}
\sup_{y\in C_p(\gamma_j)}
{P_y\left(H_R^{\theta_{jR}\eta}=0\right)}
\le \exp\{-s_p^{-\alpha}\psi_\epsilon(j,\gamma_j)R\}
\end{equation*}
for sufficiently large $R$. 
On the other hand, one can show
\begin{equation*}
\sup_{y\in C_p(\gamma_j)}
{P_y\left(X_R\in C_p(\gamma_{j+1})\right)}\\
\le \exp\set{-C_5s_p^{-\alpha} R^2
 (|\gamma_j-\gamma_{j+1}|_\infty-1)_{+}^\alpha}
\end{equation*}
for large $R$ in the same way as that for \eqref{RW-tail}.
\end{proof}
This lemma gives a control only for a fixed $\gamma$ but 
we can indeed reduce the problem to a single $\gamma$ as follows:
We have for any $\epsilon \in (0,1)$ that
\begin{equation*}
 J_v(\gamma) \ge 
(1-\epsilon) J_v(\gamma)
+ \epsilon  C_5 R
\sum_{j=0}^{v-1}(|\gamma_j-\gamma_{j+1}|_\infty-1)_{+}^\alpha.
\end{equation*}
When $p$ is so close to $1$ that 
$s_p^{-\alpha}\epsilon C_5 R^2 \ge 1$, 
for some $c>0$ depending only on $d$ and $\alpha$, 
\begin{equation*}
\begin{split}
 \sum_{\gamma\in\Gamma} 
 \exp \set{ 
 - s_p^{-\alpha}\epsilon C_5 R^2
 \sum_{j=0}^{v-1}(|\gamma_j\!-\!\gamma_{j+1}|_\infty\!-\!1)_{+}^\alpha} 
 & \le  
 \sum_{\gamma\in\Gamma} 
 \exp \set{ -\sum_{j=0}^{v-1}(|\gamma_j\!-\!\gamma_{j+1}|_\infty\!-\!1)_{+}^\alpha}\\
 &\le \exp \{cv\}.
\end{split}
\end{equation*}
Thus it follows that
\begin{equation*}
\sum_{\gamma\in\Gamma} \exp\set{-s_p^{-\alpha}J_v(\gamma)R}\le   
 \exp \set{
 -(1-\epsilon) s_p^{-\alpha}\inf_{ \gamma \in \Gamma}J_v(\gamma)R+cv}.
\end{equation*}

To conclude the proof of the upper bound, it remains to show
\begin{equation*}
\liminf_{v\to\infty}\frac{1}{v}
\inf_{\gamma\in\Gamma}J_v(\gamma)
\ge c_2(\mu_1-2\epsilon)(1-\epsilon)
\end{equation*}
almost surely. 
Without the infimum over $\gamma$, the above is a consequence 
of the law of large numbers together with Lemma~\ref{good-is-typical}.
We indeed have the tail bound
\begin{equation}
\begin{split}
&Q\big(J_v(\gamma)< c_2(\mu_1-2\epsilon)(1-\epsilon)v\big)\\
&\quad \le Q\left(\sum_{j=0}^{v-1} 
1_{\{(j,\gamma_j)\textrm{ is $\epsilon$-good}\}}
<(1-\epsilon)v\right)\\
&\quad \le 
\left(\frac{Q\left((0,0)\textrm{ is not $\epsilon$-good}\right)}
{\epsilon}\right)^{\epsilon v}
\left(\frac{1}{1-\epsilon}\right)^{{(1-\epsilon)v-1}}
\end{split}\label{Bernstein}
\end{equation}
by Bernstein's inequality. The right-hand side 
is $o(\exp\{-cv\})$ for any $c>0$ when $R$ is sufficiently
large, due to Lemma~\ref{good-is-typical}. 
We show that the infimum has no effect by counting the number
of relevant $\gamma$'s. 
Obviously we can restrict our consideration to those $\gamma$ with 
\begin{equation*}
\sum_{j=0}^{v-1}(|\gamma_j-\gamma_{j+1}|_\infty-1)_+^\alpha
\le 2(\mu_1-2\epsilon)(1-\epsilon)v{/(C_5 R)}.
\end{equation*}
Since we can find $c\ge 1$ such that $x^\alpha\le c(x-1)_+^\alpha+c$
for $x\ge 0$, the above implies
\begin{equation*}
\sum_{j=0}^{v-1}d^{-\alpha}|\gamma_j-\gamma_{j+1}|_1^\alpha
\le 2 cv
\end{equation*}
for all sufficiently large $R>0$.
We bound the number of such sequences by
\begin{equation}
\begin{split}
&\# \set{(\gamma_0=0, \gamma_1,\ldots, \gamma_v): 
\sum_{j=0}^{v-1}d^{-\alpha}|\gamma_j-\gamma_{j+1}|_1^\alpha\le 2 cv}\\
&\le \# \set{(\gamma_0=0, \gamma_1,\ldots, \gamma_v): 
\sum_{j=0}^{v-1}\sum_{i=1}^d
|\gamma_j^{(i)}-\gamma_{j+1}^{(i)}|^\alpha\le c'v},
\label{number-of-chains}
\end{split}
\end{equation}
where $\gamma_j^{(i)}$ stands for $i$-th coordinate of $\gamma_j$. 
Indeed, when $\alpha\le 1$ this holds with $c'=2cd$ as a consequence of 
the concavity of $x\mapsto x^\alpha$
and, when $\alpha>1$ with $c'=2cd^{\alpha}$ by 
$\sum_{1\le i \le d}|x_i|^{\alpha}\le (\sum_{1\le i \le d}|x_i|)^{\alpha}$.
The right-hand side of~\eqref{number-of-chains} is nothing but the 
volume of 
\begin{equation*}
\bigcup_{x\in \Z^{dv}: |x|_\alpha^\alpha\le c'v}
x+[0,1]^{dv},
\end{equation*}
where $|x|_\alpha=(\sum_{i=1}^{dv}|x_i|^\alpha)^{1/\alpha}$.
As any point $y$ in {$x+[0,1]^{dv}$
} 
satisfies
\begin{equation*}
|y|_\alpha^\alpha
\le \sum_{j=1}^{dv}2^\alpha(|x_j|^\alpha+1)
\le 2^{\alpha+2} c'v, 
\end{equation*}
the right-hand side of~\eqref{number-of-chains} is bounded by the volume
of $l^\alpha$-ball in $\R^{dv}$ with radius 
$( 2^{\alpha+2}c'v)^{1/\alpha}$, which is
known to be
\begin{equation*}
\frac{(2(2^{\alpha+2}c'v)^{1/\alpha}\Gamma(1+1/\alpha))^{dv}}
{\Gamma(1+dv/\alpha)}.
\end{equation*}
One can check by using Stirling's formula that this is only 
exponentially large in $v$. Therefore, 
with the help of~\eqref{Bernstein}, 
we find that 
\begin{equation*}
 Q\left(\inf_{\gamma\in\Gamma}J_v(\gamma)
< c_2(\mu_1-2\epsilon)(1-\epsilon)v\right)
\end{equation*}
decays exponentially in $v$ when $R$ is sufficiently large. 
\end{proof}

{\begin{remark}
\label{positivity}
We explain how to modify the above block argument to prove $\mu_p>0$
for $p\in (0,1]$. We first replace the condition~(i) of the 
$\epsilon$-good box ($\epsilon\in(0,1)$) by
\begin{equation*}
\Phi_R(\omega_p-(k,R^M x))\ge \epsilon
\end{equation*} 
and drop (ii). With this modified definition of $\epsilon$-good box,
it is simple to check that the following variant of
Lemma~\ref{good-is-typical} holds for general $p\in (0,1]$: 
\begin{equation*}
 \lim_{\epsilon\downarrow 0}\limsup_{R\to\infty}Q((k,x) \textrm{ is }
\epsilon\textrm{-good})=1.
\end{equation*}
Next we replace $J_v(\gamma)$ for $\gamma\in\Gamma$ by
\begin{equation*}
 J'_v(\gamma)=\sum_{j=0}^{v-1}\max\left\{\epsilon1_{\{(k,x)
\textrm{ is $\epsilon$-good}\}}, C_5'R(|\gamma_j-\gamma_{j+1}|_\infty
-1)_+^\alpha\right\}.
\end{equation*}
If $C_5'$ is sufficiently small, we can easily verify that any 
minimizing path $\pi_n$ for $T_n(\omega_p)$ with 
$\pi_n(jR)\in C_p(\gamma_j)$ ($0\le j\le n/R$) has  passage time 
larger than $J'_v(\gamma)$. 
Therefore we get
\begin{equation*}
T_n(\omega_p)\ge \inf_{\gamma\in\Gamma} J'_{[n/R]-1}(\gamma)
\end{equation*}
and, when $R\in\N$ is chosen sufficiently large and $\epsilon$ small, 
we have
\begin{equation*}
 \liminf_{v\to\infty}\frac1v\inf_{\gamma\in\Gamma} J'_v(\gamma)>0
\end{equation*}
in exactly the same way as above. 
\end{remark}
}

\section*{Appendix}
We provide a proof of~\eqref{lb} for completeness. 
We consider $d=1$ case first since the other case will 
reduce to it. 
Set
\begin{equation*}
 \mathcal{L}=\{(m,x)\in \N \times \Z: m+x\in 2\Z\}. 
\end{equation*}
For $R>0$ and $(m,x)\in \mathcal{L}$, we say 
$(m,x)$ is {\em open} if there exists a $y_m\in Rx+(-R,R)\cap\Z$
such that $\eta(m,y_m)=0$. 
It is easy to see that
\begin{equation*}
 Q((m,x)\textrm{ is open})
 \to 1 
\end{equation*}
as $R\to\infty$. Thus when $R$ is large, 
the directed site percolation on $\mathcal{L}$
is supercritical and we can find a percolation point
$(1,x)\in \mathcal{L}$. 
This implies that there exists a path
$\{(k,y_k)\}_{k\ge 1}$ satisfying
\begin{equation*}
 \eta(k,y_k)=0\textrm{ and }|y_{k+1}-y_{k+2}|\le 3R
\end{equation*}
for all $k\ge 1$. Then it follows that
\begin{equation*}
\begin{split}
\liminf_{n\to\infty}\frac{1}{n}\log Z_n^{\eta,-\infty}
&\ge \liminf_{n\to\infty}\frac{1}{n}\log P(X_k=y_k\textrm{ for all }
 k\le n)\\
&\ge -c_2 3^\alpha R^\alpha.
\end{split}
\end{equation*}
For the case $d\ge 2$, we have 
\begin{equation*}
 Z_n^{\eta,-\infty}\ge P(H_n^\eta=0 \textrm{ and } 
X_k \in \Z\times \{0\}^{d-1}\textrm{ for all }1\le k\le n)
\end{equation*}
and the right-hand side can be bounded from below in the same way 
as for $d=1$.

\section*{Acknowledgements}
The first author was partially supported by CNRS, UMR 7599.
The second author was supported by JSPS KAKENHI Grant Number 24740055. 
The fourth author was supported by JSPS KAKENHI Grant Number 25400136.

\newcommand{\noop}[1]{}


\end{document}